\documentclass[12pt,a4paper]{amsart}
 \pdfoutput=1
 
 \usepackage{amsmath,amssymb,amsthm}  
 \usepackage{mathtools}
 \usepackage{thmtools}
 \usepackage{float} 
 
 \usepackage{xcolor} 	
 \usepackage{hyperref}
 \hypersetup{
 	colorlinks,
     linkcolor={red!60!black},
     citecolor={green!60!black},
     urlcolor={blue!60!black},
 }
 \usepackage[abbrev, msc-links]{amsrefs}
 \usepackage{comment}
 \usepackage{todonotes}
 
 \usepackage[utf8]{inputenc}
 \usepackage[T1]{fontenc}
 \usepackage{lmodern}
 \usepackage[babel]{microtype}
 \usepackage[english]{babel}

 \usepackage{geometry}
 \usepackage{enumitem}

\theoremstyle{plain}
\newtheorem{thm}{Theorem}[section]

\newtheorem{claim}[thm]{Claim}
\newtheorem{cor}[thm]{Corollary}
\newtheorem{lem}[thm]{Lemma}

\newtheorem{obs}[thm]{Observation}

\newtheorem*{thm*}{Theorem}
\newtheorem*{cor*}{Corollary}

\theoremstyle{definition}

\newtheorem{exm}[thm]{Example}

\title{Cutting a cake for infinitely many guests}

\author{Zsuzsanna  Jankó}
\thanks{Jankó is grateful for the support of  NKFIH 
OTKA-K128611}
\address{Zsuzsanna  Jankó,  Department of Operations Research and Actuarial Sciences,  Corvinus University of Budapest,  
Budapest, Hungary;
Institute of Economics, Centre for Economic and Regional Studies, 
Budapest, Hungary}
\email{zsuzsanna.janko@uni-corvinus.hu, janko.zsuzsanna@krtk.hu}
\author{Attila Jo\'{o}}
\thanks{Joó would like to thank the generous support of the Alexander 
von Humboldt Foundation and NKFIH 
OTKA-129211}
\address{Attila Jo\'{o},
Department of Mathematics, University of Hamburg,  Hamburg, Germany; Set theory, logic and topology research division, 
Alfr\'{e}d R\'{e}nyi Institute of Mathematics,  Budapest, Hungary}
\email{attila.joo@uni-hamburg.de, jooattila@renyi.hu}

\keywords{fair division, cake cutting, last diminisher}
\subjclass[2020]{Primary: 91B32, 91A07  Secondary: 68W30} 
\begin{document}

\begin{abstract}
Fair division with unequal shares is an intensively studied resource allocation problem. For  $ i\in [n] $, let $ \mu_i $ be an 
atomless probability measure on the measurable space $(C,\mathcal{S})  $ and let $ t_i  $ be positive numbers (entitlements) 
with $ \sum_{i=1}^{n}t_i=1 $. A fair division  is a partition of  $ C $  into sets $ S_i\in \mathcal{S} $ with $ \mu_i(S_i)\geq t_i $ 
for every $ i\in [n] $. 

We introduce new algorithms to solve the fair division problem with  irrational entitlements. They are based on the classical Last 
diminisher technique and we believe that they are simpler than the known methods. Then we show that a fair division always 
exists even for infinitely many players.
\end{abstract}
\maketitle

\section{Introduction}
Cake cutting is a metaphor of  the distribution of some  inhomogeneous continuos goods  and is intensively 
investigated by not just 
mathematicians but economists and political scientists as well. The preferences of the players $ P_i $ involved in the sharing 
 are usually  represented as  atomless  probability
measures $ \mu_i $ defined on a common  $ \sigma $-algebra $ \mathcal{S}\subseteq \mathcal{P}(C) $ of the possible `slices' of 
the `cake' $ C $. One option of how a division  can be ``good'' is \emph{proportionality}. This 
means that each of the $ n $ players gets at least one $ n $th of the cake  according to their own measurement, i.e. $ 
C=\bigsqcup_{i=1}^{n}S_i $ with $ \mu_i(S_i) \geq \frac{1}{n} $. The division is called \emph{strongly proportional} if all 
these inequalities are strict. For $ n=2 $ a proportional division can be found by the so called  ``Cut and choose''  procedure. This 
was used by Abraham and Lot in the Bible to share Canaan.   Abraham divided Canaan into 
two parts which have equal value for him and then Lot chose his favourite among these two parts  leaving Abraham the other one. 
For a general $ n $,  Steinhaus challenged his students Banach and Knaster to find a solution that they successfully accomplished 
by 
developing the so called  ''Last Diminisher'' procedure (see \cite{steihaus1948problem}). In this method $ P_1 $ picks a slice $ 
T_1 $ with $\mu_1(T_1) 
=\frac{1}{n} $. If $ \mu_2(T_1)>\frac{1}{n} $, then $ P_2 $ diminishes $ T_1 $ in the sense that  he takes an $ T_2\subseteq 
T_1 $ with $\mu_2(T_2) =\frac{1}{n} $, otherwise he lets $ T_2:=T_1 $. They proceed similarly and slice $ T_n $ is allocated  
to the player who lastly diminished or to $ P_1 $ if nobody did so.  Then the remaining cake worth at least $ \frac{n-1}{n} $ for 
each 
of the remaining $ n-1 $ players and they can continue using the same protocol. 

A natural extension 
of the concept of proportional division is the so called ``fair division with unequal shares''.  In this variant there 
are entitlements $ t_i >0 $  associated to the 
players satisfying $ \sum_{i} t_i=1 $. A division is called (strongly) \emph{fair} if the slice $ S_i $ given to player $P_i $  worths 
for him 
at least (more than) $ t_i $, i.e. $ \mu_i(S_i) \geq t_i  $ ($ \mu_i(S_i) > t_i  $) holds for each $ i $. If all of these entitlements are 
rational numbers, 
say $ \frac{p_1}{q},\dots, \frac{p_n}{q} $, then a fair division according to them can be reduced to a proportional division 
problem for $ \sum_{i=1}^{n} p_i $ players where measure $ \mu_i $ is ``cloned'' to $ p_i $ copies.  In the presence of 
irrational 
entitlements such a ``player-cloning`` argument is no more applicable. 

Several finite procedures were developed to find a (strongly) fair 
division allowing irrational entitlements. In the special case of the problem where $ C=[0,1] $, $ \mathcal{S} $ is the Borel $ \sigma $-algebra and 
the measures $ \mu_i $ are 
absolute continues, Shishido and Zeng developed and algorithm in \cite{shishido1999mark}. In their protocol the players choose intervals that 
worth the same and exchange these intervals among each other exploiting the possible differences of their evaluations. A more recent 
algorithm in the same model but based on completely different ideas was given by Cseh and Fleiner in \cite{cseh2020complexity}. In a general 
step they reduce  the problem to two sub-problems in one of which the number of players is smaller by one while in the other all the 
entitlements are rational and the number of players remains the same. 

Our first contribution (Section \ref{sec: Last dim}) is two  procedures solving the fair division problem with potentially irrational entitlements 
which we believe are simpler 
than the known methods. We keep working in the general settings we have already introduced which originated from Barbanel  
(see \cite{barbanel1995game}). Our aim is to demonstrate that `Last diminisher'-type of ideas are already powerful enough to design simple finite 
procedures.  We provide two 
algorithms  both of which 
solves the 
problem in finitely many steps. The  first one reduces  the problem  to another one in which either 
the number of players is smaller by one or all the entitlements are rationals and the  number of players is the same. This can be considered a direct 
improvement of the algorithm given in Section 7 of \cite{cseh2020complexity}.  Taking 
rational numbers from  non-degenerate intervals (which was used in the 
algorithms given in \cites{woodall1986note, cseh2020complexity}) is necessary in this procedure. In our second algorithm not even such a rational 
approximation is 
needed.

It was shown in \cite{dubins1961cut} based on Lyapunov’s 
theorem that if not all the measures are identical, then a strongly 
 proportional division exists. A constructive proof was obtained later in \cite{woodall1986note} which was then further 
 developed for the case of unequal shares (i.e. strong fairness) in \cite{barbanel1995game}. We point out in Section 
 \ref{sec: strong fair}  that the 
 strongly fair division problem (for potentially infinitely many players) can be actually reduced to the fair division problem in a 
 completely elementary way, which 
reduction we need later. 
 
 In the last section (Section \ref{sec: inf}) we 
consider the (strongly) fair division 
problem for infinitely many players. Rational entitlements do not  make this problem easier since representing them with a 
common 
denominator is  impossible in general. Since the entitlements sum up to $ 1 $, they must converge to $ 0 $,
thus extension of protocols in which 
one need to start with the smallest positive entitlement (like the one given in  \cite{cseh2020complexity}) is 
problematic. By Last Diminisher-type of methods we are facing in addition the  difficulty that diminishing infinitely often might 
be necessary 
in which case 
no ``last diminisher'' exists, moreover, we may end up with the empty set as a limit of the iterated trimmings. Eliminating one 
player and using induction for the rest is also not applicable for obvious reasons. Although the so called Fink protocol (see 
\cite{fink1964note}) can be considered as such a player-eliminating recursive algorithm,  it inspired our procedure that finds a 
fair division for infinitely many players:

\begin{restatable}{thm}{cake}\label{t: main cake}
Assume that  $ (C, \mathcal{S}) $ is a measurable space and for $ i\in \mathbb{N} $,  $ 
\mu_i $ is an atomless probability
measure  defined on $ \mathcal{S} $ and $ t_i $ is a positive number such that $ \sum_{i=0}^{\infty}t_i=1 $. Then there 
is a partition $ C=\bigsqcup_{i=0}^{\infty}S_i $ such that 
$ 
S_i\in \mathcal{S} $ with $ \mu_i(S_i)\geq t_i $ for each $ i\in \mathbb{N} $. Furthermore, if not all the $ \mu_i $ are 
identical, 
then  `$ \mu_i(S_i)\geq t_i $' can be strengthened to `$ \mu_i(S_i)> t_i $' for every $ i\in \mathbb{N} $.
\end{restatable}

Let us mention  that cake cutting problems have a huge literature and this particular model and notion of fairness that we consider 
is only a tiny fragment of it. About the so called  exact, envy-free and equitable divisions (none of which are extendable to 
infinitely many players for obvious reasons) and the corresponding  
existence results  
a brief but informative survey can be found in \cite{cheze2017existence}.  For a more general picture about  this field, 
including completely different mathematical 
models of the problem, we refer 
to \cite{barbanel2005geometry}, \cite{brams1996fair}, \cite{procaccia2015cake} and \cite{robertson1998cake}.

\section{`Last Diminisher'-type of procedures for fair division with irrational entitlements}\label{sec: Last dim}
Our aim is to find a fair division $ S_1,\dots, S_n $ for players $ P_1,\dots, P_n $ with respective atomless probability measures $ 
\mu_1,\dots, \mu_n 
$ and (potentially irrational) entitlements $0<t_1\leq t_2\leq 
\dots \leq t_n<1 $ where $ \sum_{i=1}^{n}t_i=1 $.

As it is standard in the cake cutting literature, algorithms use certain queries. We allow the following operations.
\begin{itemize}
\item The four basic arithmetical operations and comparison on $ \mathbb{R} $.
\item The set operations on $ \mathcal{S} $.
\item Computing $ \mu_i(S) $ for some $ i\in[n] $ where slice $ S $ is obtained in a previous step.
\item  Cutting a slice $ S'\subseteq S $  with $ \mu_i(S)=\alpha $ for an $ i\in [n]$ and $ \alpha\in[0,\mu_i(S)] $ where 
either $ S=C $ or $ S $ is obtained in a previous step.\footnote{It is well-defined because $ \mu_i $ is atomless (see  
\cite{lorenc2013darboux}*{Theorem 5}).}
\end{itemize}

\subsection{Algorithm I}
Player $ P_1 $ picks 
some $ T_1$ with $ \mu_1(T_1) =t_1$. If $ T_{i} $ is already defined for some $ i<n $, we let $ 
T_{i+1}:=T_{i} $ if $ \mu_{i+1}(T_{i})\leq t_1 $ and we define $ T_{i+1} $ to be a subset of $ T_{i} $ with $ 
\mu_{i+1}(T_{i+1})=t_1 $ if $ \mu_{i+1}(T_{i})> t_1 $. After the recursion is done, $ \mu_i(T_{n}) \leq t_1 $ holds for each $ 
i $ 
and 
there is equality for at least 
one index.

If $ \mu_{1}(T_{n})=t_1 $, then we let $ S_1:=T_{n} $ and remove player $ P_1 $ 
 from the process. Since the rest of 
the cake worth at least $ 1-t_1
$ for all the  players, dividing it fairly with respect to the entitlements  $ \frac{t_i}{1-t_1} $ for $1<i\leq n $ leads to a 
fair division. Thus we invoke the algorithm for this sub-problem with less players. 

If $ \mu_{1}(T_{n})<t_1 $, then there must be a player who diminished the slice during the recursion. Let  $ k $ be the largest 
index for which $ P_k $ is such a player. 
We allocate $ T_{n} 
$ to $ P_k $ but we do  not remove $ P_k $ from the process unless  $ t_1=t_k $. In order to satisfy $ 
P_k $, he needs to get at least the $t'_k:= \frac{t_k- t_1}{\mu_k(C\setminus T_{n})} $ fraction of the rest of the cake $ 
C\setminus T_{n} $ according to his measure $ \mu_k $, while for $ i\neq k $ player $ P_i $ should get at least the fraction $ 
t'_i:= 
\frac{t_i}{\mu_i(C\setminus T_{n})}$ of $ C\setminus T_{n} $ w.r.t. $ \mu_i $.  As we already noticed $ \mu_i(T_{n})\leq t_1 
$ and hence  $ 
\mu_i(C\setminus 
T_{n}) \geq 1-t_1 $ for every $ i $, furthermore,  the inequality is strict for $ i=1 $ in this branch of the case distinction. 
Therefore  

\[\sum_{i=1}^{n}t'_i<
\frac{t_k- t_1}{1-t_1}+\sum_{i\neq k}\frac{t_i}{1-t_1}=\frac{(\sum_{i=1}^{n}t_i)-t_1}{1-t_1}=1. \]
Thus we can pick  rational numbers $ t''_i > t_i' $ with $ \sum_{i\leq n}t''_i=1 $. Finally, we use a subroutine  to divide $ 
C\setminus T_{n}  $ fairly among the players w.r.t. the rational entitlements $ t''_i $ to obtain a strongly fair division  for the 
original problem.

\subsection{Algorithm II}

In this algorithm no `rounding up to rationals' is necessary. We shall  make  several 
rounds and in each of them allocate a slice chosen in a `Last diminisher' manner. The satisfied players are dropping out of the 
process. The algorithm itself is quite simple in this case as well but the proof of the correctness is somewhat more involved.

For $i\in [n]  $, we denote by $ S^{m}_i 
$   the portion  allocated  to player $ P_i $ at the beginning of round $ m $.  We set $ S^{0}_i=\emptyset $ for 
every $ i $. The rest of the cake is $C_m:= C\setminus 
\bigcup_{i=1}^{n}S_i^{m}$.  We also 
have improved entitlements $ t_i^{m} $  where $ 
t_i^{0}:=t_i $. We say that player $ P_i $ is satisfied at the beginning of round $ m $ if $ t_i\leq \mu_i(S_i^{m}) $. Let us 
define $ I_m $ as the set of indices of the 
players that are unsatisfied at the beginning of round $ m $, i.e.
\[ I_m:=\{ i\in [n]:\ t_i>\mu_i(S_i^{m}) \}. \] 
If $ I_m=\emptyset $, then the process terminates and the sets $ S_i^{m} $ for $ i\in[n] $ form a fair division. If 
$ I_m\neq \emptyset $, then the algorithm does the following. Let
\[ c_m:= \min_{i\in I_m} \frac{t_i^{m}-\mu_i(S_i^{m})}{\mu_i(C_m)} \]
 and let $ i_m\in 
I_m $ the smallest index where this minimum is taken. Then 
player $ P_{i_m} $ 
picks a $ T_1^{m}\subseteq C_m $ with $ \mu_{i_m}(T_1^{m})= t_{i_m}^{m}-\mu_{i_m}(S_{i_m}^{m})$.  After this, 
players $ P_i $ for $ i\in I_m\setminus \{ i_m 
\} $ consider the (actual) 
slice one by one and diminish it or keep unchanged in the following way. If $ T_k^{m} $ is defined, 
$ k<\left|I_m\right| $ and $ \ell $ is the $ k $th smallest element of $ I_m\setminus \{ i_m \}$, then let

\[ 
T_{k+1}^{m}:= \begin{cases} T_k^{m} &\mbox{ if }  \frac{\mu_\ell(T_k^{m})}{\mu_\ell(C_m)} \leq 
c_m \\
S \text{ with }S\subseteq T_k^{m} \text{ and }\frac{\mu_\ell(S)}{\mu_\ell(C_m)}=c_m & \mbox{ if }  
\frac{\mu_\ell(T_k^{m})}{\mu_\ell(C_m)} >
c_m.
\end{cases} \]
Eventually they obtain 
$T^{m}_{\left|I_m\right|}=:R_m $ for which
\begin{equation}\label{eq1}
 \frac{\mu_i(R_m)}{\mu_i(C_m)}\leq \frac{t_{i_m}^{m}-\mu_{i_m}(S_{i_m}^{m})}{\mu_{i_m}(C_m)}=c_m
\end{equation} for every $ i\in 
I_m $ and there is equality for at least one index. Let $ 
j_m:=i_m $ if there is equality at (\ref{eq1}) for $ i_m $ and  let $ j_m $ be the smallest index in $ I_m $ for which we have 
equality if the 
inequality is strict for $ i_m $. We allocate $R_m $ to player $ 
P_{j_m} $, formally $ S_{j_m}^{m+1}:=S_{j_m}^{m}\cup R_m  $ and  $ S_i^{m+1}:=S_i^{m} $ for $ i\in [n]\setminus \{ 
j_m \} $. For $ i\in I_{m+1} $ let
\[ t_{i}^{m+1}:=\mu_i(S_i^{m+1})+ \frac{t_{i}^{m}-\mu_i(S_i^{m+1})}{\sum_{j\in I_{m+1}  
}\frac{t_{j}^{m}-\mu_j(S_j^{m+1})}{\mu_j(C_{m+1})}},\] 
which completes the description of the general step of the algorithm.  

Let us  
shade some more light  on the running of the algorithm and on the ideas behind the formal definitions by a concrete example:
\begin{exm}
Let the cake be the unit interval $ [0,1] $ and the slices are defined to be the Borel subsets. We have $ 3 $ players with respective entitlements $ 
t_1=\frac{1}{2}, 
t_2=\frac{1}{3}\text{ and } t_3=\frac{1}{6} $. The measure $ \mu_1 $ is the uniform distribution on $ [0,\frac{1}{2}] $, $ \mu_2 $ is the same 
but on $ [\frac{1}{2}, 1] $ and $ \mu_3 $ is the uniform distribution on the whole cake $ [0,1] $.

Then the constant $ c_0 $ is simply the smallest entitlement $ \frac{1}{6} $ and $ i_0=3 $. Player $ P_3 $ cuts off a slice which could be $ 
T^{0}_1=[0, \frac{1}{6}]$. Player $ P_1 $ diminishes this slice,  he cuts off for example $ T^{0}_2=[0,\frac{1}{12}] $ (this worths $ 
\frac{1}{6} $ for him ). Since $ \mu_2(T^{0}_2)=0 $, player $ P_2 $ does not change this slice, i.e. $ T^{0}_3=T^{0}_2 $. We have $ j_0=1 $ 
and 
allocate $ R_0:=T^{0}_3 
$ to $ P_1 $, more precisely $ S^{1}_1:=R_0 $ and $ S^{1}_2:=S^{1}_3:=\emptyset $.

Now $ P_1 $ still needs $ \frac{1}{2}-\frac{1}{6}=\frac{1}{3} $. This is the $ \frac{\frac{1}{3}}{\frac{5}{6}}=\frac{2}{5} $ fraction of the 
remaining cake according to his own measure. The removed part has no value for $ P_2 $, thus he still wants the $ \frac{1}{3} $ fraction of the 
remaining 
part.  Finally, $ P_3 $ wants the $ \frac{\frac{1}{6}}{\frac{11}{12}}=\frac{2}{11} $ fraction of the rest according to his  measure. Now the 
algorithm ``scales'' the ratios $\frac{2}{5},  \frac{1}{3} $ and $ 
\frac{2}{11} $ in order to sum up to $ 1 $, norms the 
measures 
to be probability measures on the remaining 
cake and does the same that it initially did.
\end{exm}

We proceed by proving that the steps of the algorithm are well-defined and it always terminates:
\begin{lem}\label{lem: steps done eq strict}
The steps of  Algorithm II  can be done and it  maintains the equation
\begin{equation}\label{eq2}
\sum_{i\in I_m} \frac{t_i^{m}-\mu_i(S_i^{m})}{\mu_i(C_m)}=1 
\end{equation}
as well as the inequalities $ t_i^{m}\geq t_i $  for every $ i\in [n] $.
\end{lem}  
\begin{proof}
We use induction on $ m $. For $ m=0 $, (\ref{eq2})  says $ \sum_{i=1}^{n}t_i=1 $ which we assumed and $ t_i^{0}=t_i $ 
by definition. Suppose we are at 
the beginning of round $ m $ and  (\ref{eq2}) holds so far and $ t_i^{m}\geq t_i $ for  $ i\in [n] $. By the definition of $ I_m $ 
and by $ t_i^{m}\geq t_i $, the summands at 
 (\ref{eq2}) are all positive, thus we have $ 
\frac{t_{i_m}^{m}-\mu_{i_m}(S_{i_m}^{m})}{\mu_{i_m}(C_m)}\leq1 $. Therefore $  
t_{i_m}^{m}-\mu_{i_m}(S_{i_m}^{m})\leq\mu_{i_m}(C_{m})$ and hence
there is indeed a $ T_1^{m}\subseteq C_m $ with $ \mu_{i_m}(T_1^{m})= t_{i_m}^{m}-\mu_{i_m}(S_{i_m}^{m})$. If we know that the 
entitlements $ t_i^{m+1} $ for $ i\in I_{m+1} $ are well-defined (i.e. no zeros in the denominators), then a direct calculation shows that they sum 
up 
to $ 1 $.
\begin{obs}\label{obs: satisfy player}
If $ j_m=i_m 
$, then player $ P_{i_m} $ will be satisfied  after round $ m $ because in this case we 
have equality at (\ref{eq1}) for $ i_m $ and $ t_{i_m}^{m} \geq t_{i_m} $.
\end{obs}
\noindent If $ I_m=\{ i_m \} $, then we must have $ j_m=i_m $. It follows by Observation \ref{obs: satisfy player} that $ 
I_{m+1}=\emptyset $ and therefore the algorithm terminates, thus in this case there is nothing more to prove.

Suppose 
that $ \left|I_m\right|>1 $. Then the right side of (\ref{eq1}) is strictly smaller than $ 1 $ because it is one of the summands at 
(\ref{eq2}) which are: all positive, there are at least two of them and they sum up to $ 1 $.
By subtracting both sides of (\ref{eq1}) from $ 1 $ and taking the reciprocates we obtain  \begin{equation}\label{eq3}
 \frac{\mu_i(C_{m})}{\mu_i(C_{m+1})}\leq 
\frac{1}{1- \frac{t_{i_m}^{m}-\mu_{i_m}(S_{i_m}^{m})}{\mu_{i_m}(C_m)}}
\end{equation}
for $ i\in I_m $. In particular $ \mu_j(C_{m+1})>0 $ for $ j\in I_{m+1} $. Since $ t_{j}^{m} \geq t_j>\mu_j(S_j^{m+1}) $, the entitlements $ 
t_i^{m+1} $ are 
indeed well-defined.
\begin{claim}\label{claim: strict}
We have
\[ \sum_{i\in I_{m+1}  
}\frac{t_{i}^{m}-\mu_i(S_i^{m+1})}{\mu_i(C_{m+1})}\leq 1 \]   as well as  $ t_i^{m+1}\geq t_i^{m} 
$ for every $ i\in I_{m+1} $, moreover,  all of 
these  inequalities are strict if $ j_m\neq i_m $. 
\end{claim}
\begin{proof}
\begin{align*}
&\sum_{i\in I_{m+1}  
}\frac{t_{i}^{m}-\mu_i(S_i^{m+1})}{\mu_i(C_{m+1})}\stackrel{*}{\leq}\ \sum_{i\in I_{m}  
}\frac{t_{i}^{m}-\mu_i(S_i^{m+1})}{\mu_i(C_{m+1})}=\sum_{i\in I_{m}  
}\frac{t_{i}^{m}-\mu_i(S_i^{m+1})}{\mu_i(C_{m})} \cdot \frac{\mu_i(C_{m})}{\mu_i(C_{m+1})}
\stackrel{(\ref{eq3})}{\leq}\\ &\sum_{i\in I_{m}  
}\frac{t_{i}^{m}-\mu_i(S_i^{m+1})}{\mu_i(C_{m})} \cdot \frac{1}{1- 
\frac{t_{i_m}^{m}-\mu_{i_m}(S_{i_m}^{m})}{\mu_{i_m}(C_m)}}\stackrel{**}{\leq}\left[ \left( \sum_{i\in I_{m}  
}\frac{t_{i}^{m}-\mu_i(S_i^{m})}{\mu_i(C_{m})}\right)-\frac{\mu_{j_m}(R_m)}{\mu_{j_m}(C_m)}\right]   \cdot 
\frac{1}{1- \frac{t_{i_m}^{m}-\mu_{i_m}(S_{i_m}^{m})}{\mu_{i_m}(C_m)}}\\
&\stackrel{***}{=}
\left[ 1- \frac{t_{i_m}^{m}-\mu_{i_m}(S_{i_m}^{m})}{\mu_{i_m}(C_m)} \right] \frac{1}{1- 
\frac{t_{i_m}^{m}-\mu_{i_m}(S_{i_m}^{m})}{\mu_{i_m}(C_m)}}=1.
\end{align*} 
$ *\  I_{m+1}\subseteq I_m$ and the summands are non-negative,\\ 
$ **\  \mu_{j_m}(S_{j_m}^{m+1})=\mu_{j_m}(S_{j_m}^{m})+ 
\mu_{j_m}(R_m) $ because $ S_{j_m}^{m+1} $ is the disjoint union 
of $ S_{j_m}^{m} $ and $ R_m $, furthermore, $ S_i^{m+1}=S_i^{m} $ for $ i\in I_m\setminus \{ j_m \} $,\\ 
$ *** $ (\ref{eq2}) and there is equality at (\ref{eq1}) for $ j_m $.\\ 

The overestimation of $ 
\frac{\mu_{i_m}(C_{m})}{\mu_{i_m}(C_{m+1})} $ via (\ref{eq3})  is strict if 
$ i_m\neq j_m $.  The part about the  inequalities  $ t_i^{m+1}\geq t_i^{m} 
$ follows directly from the already proved part and the definition of $ 
t_i^{m+1} $.
\end{proof}
\end{proof}
\begin{lem}
Algorithm II terminates after finitely many steps.
\end{lem}
\begin{proof}
Suppose for a contradiction that  the algorithm does not terminate for $ \mu_1,\dots, \mu_n $ and $ t_1,\dots, t_n $. Let $ k $ 
be 
the smallest number for which $ I_k=I_m $ for every $ m>k $. Then $ j_k\neq i_k $ since otherwise we had $ 
I_{k+1}=I_k\setminus \{ i_k \}\subsetneq I_k$ (see Observation \ref{obs: satisfy player}). By Claim \ref{claim: strict} this 
implies $ t_i^{k+1}>t_i^{k}\geq t_i $ for every $ 
i\in I_k $. Let $ (m_\ell)_{\ell\in \mathbb{N}} $ be a strictly increasing sequence of natural numbers with $ m_0>k $ such that there are $ 
i^{*}, 
j^{*}\in I_k $ with $ i_{m_\ell}=i^{*} $ and $ j_{m_\ell}=j^{*} $ for every $ \ell $. There cannot be a $ \varepsilon>0 $ 
such 
that 
$ \mu_{j^{*}}(R_{m_\ell})\geq \varepsilon $ for infinitely many $ \ell $ because then $ P_{j^{*}} $ would be eventually 
satisfied and removed from the  process, 
contradicting the definition of $ k $. Thus $ \lim_{\ell \rightarrow \infty} \mu_{j^{*}}(R_{m_\ell})=0 $. Since there is 
equality 
for $ 
j^{*} $ at (\ref{eq1}) for each $ m_{\ell} $, we know that 
\[ \mu_{j^{*}}(R_{m_\ell})= \left[ t_{i^{*}}^{m_\ell}-\mu_{i^{*}}(S_{i^{*}}^{m_\ell})\right]   
\frac{\mu_{j^{*}}(C_{m_\ell})}{\mu_{i^{*}}(C_{m_\ell})}. \]
If $\liminf_{\ell \rightarrow \infty} t_{i^{*}}^{m_\ell}-\mu_{i^{*}}(S_{i^{*}}^{m_\ell})=0 $, then $ 
\mu_{i^{*}}(S_{i^{*}}^{m_\ell}) \geq t_{i^*} $ for some $ \ell $ because  $t_{i^{*}}^{m_0}>t_{i^*}  $ and  $ 
t_{i^{*}}^{m_\ell} $ is increasing in $ \ell $, a contradiction. Therefore we must have $ \lim_{\ell \rightarrow 
\infty}\frac{\mu_{j^{*}}(C_{m_\ell})}{\mu_{i^{*}}(C_{m_\ell})}=0 $. Since $ \mu_{i^{*}}(C_{m_\ell})\leq  
\mu_{i^{*}}(C_{m_0})$, this implies $ \lim_{\ell \rightarrow \infty}\mu_{j^{*}}(C_{m_\ell})=0 $. But then it follows from 
(\ref{eq2}) that $\lim_{\ell 
\rightarrow \infty} t_{j^{*}}^{m_\ell}-\mu_{j^{*}}(S_{j^{*}}^{m_\ell})=0 $. As earlier with $ i^{*} $, this implies that 
player 
$ P_{j^{*}} $ will be eventually satisfied, which is a contradiction.
\end{proof}
 \section{From fairness to strong fairness, an elementary approach}\label{sec: strong fair}
\begin{lem}\label{l: t't''}
Assume that  $ (C, \mathcal{S}) $ is a measurable space, $ I $ is a countable index set, and for $ i\in I $,  $ 
\mu_i $ is an atomless probability
measure  defined on $ \mathcal{S} $ and $ t_i $ is a positive number such not all the $ \mu_i $ are identical and  $ 
\sum_{i\in I}t_i=1 $. Then there 
is a partition $ C=C'\sqcup C'' $ and $ 
t_i', 
t_i'' > 0$ with $ \sum_{i\in I}t'_i=\sum_{i\in I}t''_i =1$   such that  $ 
t'_i \cdot 
\mu_i(C')+t''_i \cdot 
\mu_i(C'')>t_i $  for each $ i\in I $. 
\end{lem}
\begin{proof}
Suppose that $ j,k\in I $ and $ C'\in \mathcal{S} $ such that $ \mu_j(C')<\mu_k(C') $.  It is enough to find $ s'_i, 
s''_i>0 $ with $ \sum_{i\in I}s'_i, \sum_{i\in I}s''_i<1 $ and $ 
s'_i \cdot 
\mu_i(C')+s''_i \cdot 
\mu_i(C'')=t_i $ for every $ i\in I $ because then \[ t'_i:=\frac{s'_i}{ \sum_{\ell\in I}s'_\ell}\text{ and 
}t''_i:=\frac{s''_i}{ \sum_{\ell\in I}s''_\ell} \] are as desired. We are 
looking for $ \varepsilon, \delta>0 $ for which the definitions
\begin{itemize}
\item $ s'_j:=t_j-\varepsilon$
\item $ s''_j:= t_j+ \varepsilon\cdot \frac{\mu_j(C')}{\mu_j(C'')} $
\item $ s'_k:= t_k+\delta \cdot \frac{\mu_k(C'')}{\mu_k(C')}$
\item $ s''_k:=t_k-\delta $
\item $ s_i'':= s_i':=t_i$ for $ i\in \mathbb{N}\setminus \{ j,k \} $
\end{itemize}
are suitable. Note that whatever  $ \varepsilon $ and $ \delta $ we choose,  $s'_i \cdot 
\mu_i(C')+s''_i \cdot \mu_i(C'')=t_i $ will hold  for each $ i\in \mathbb{N} $. Thus the requirements $ s_i', s_i''>0 $ and  $ 
\sum_{i\in I}s'_i, \sum_{i\in I}s''_i<1 $ mean for $ \varepsilon $ and $ 
\delta $ that they satisfy
\begin{align*}
\varepsilon &\in (0,t_j)\\
\delta &\in (0, t_k)\\
\varepsilon&>\delta\cdot \frac{\mu_k(C'')}{\mu_k(C')}\\
\delta&>\varepsilon\cdot \frac{\mu_j(C')}{\mu_j(C'')}
\end{align*} 
If $ \mu_j(C')=0 $, then the last inequality is redundant and the existence of a solution is straightforward. Otherwise the last two 
inequalities demand \[  \frac{\mu_k(C'')}{\mu_k(C')}<\frac{\varepsilon}{\delta} < \frac{\mu_j(C'')} 
{\mu_j(C')}.\] Since $ \frac{\mu_k(C'')}{\mu_k(C')}<\frac{\mu_j(C'')}{\mu_j(C')}  $ follows from $ \mu_j(C')<\mu_k(C') $, the 
desired $ \varepsilon $ and $ \delta $ exist in this case as well.
\end{proof}

Let $ \mu_i' $ be  the restriction of $ \frac{\mu_i}{\mu_i(C')} $ to $ \mathcal{S}\cap \mathcal{P}(C') $ if $ \mu_i(C')\neq 0 $ 
and an arbitrary  atomless probability measure on $ \mathcal{S}\cap \mathcal{P}(C') $ if $ \mu_i(C')= 0 $. We define $ \mu_i'' $ 
analogously with respect to $ C'' $.
\begin{cor}\label{cor: strong elim}
Assume the settings of Lemma \ref{l: t't''}. If $ \{ S'_i:\ i\in I \} $ is a fair division with respect to $ \mu_i', t_i'\ (i\in 
I) $  and 
$ \{ S''_i:\ i\in 
I \} $ is a fair divisions with respect to $ \mu_i'', t_i''\ (i\in I)  $,  then for $ S_i:=S_i'\cup S''_i $,  $ \{ S_i:\ 
i\in I \} $ 
is a strongly fair division
with respect to  $ \mu_i, t_i\ (i\in I)  $.
\end{cor}
\begin{proof}
We have $ \mu_i(S'_i)\geq t_i' \cdot \mu_i(C') $ and $ \mu_i(S''_i)\geq t_i'' \cdot \mu_i(C'') $ by fairness, thus by Lemma 
\ref{l: t't''} \[ \mu_i(S_i)=\mu_i(S_i'\sqcup S_i'')=\mu_i(S_i')+\mu_i(S_i'')\geq t_i' \cdot \mu_i(C') +t_i'' \cdot \mu_i(C'') >t_i. \]
\end{proof}

\section{Existence of a fair division for infinitely many players}\label{sec: inf}
We repeat the theorem here for convenience.
\cake*

\begin{proof}
Without loss of generality we may look for a sub-partition instead of a 
partition, i.e. we can  relax `$ C=\bigsqcup_{i=0}^{\infty}S_i $' 
to `$ C\supseteq\bigsqcup_{i=0}^{\infty}S_i $' since the remaining surplus part of the cake  can be given  to anybody.  The last 
sentence of 
Theorem \ref{t: main cake} follows from the rest of it via Corollary \ref{cor: strong elim}.

For $ n\in \mathbb{N} $, we let $ t_0^{n}, 
t_1^{n},\dots, t_{n}^{n} $ to be the first $ n+1 $ entitlements scaled to sum up to $ 1 $, i.e.   \[ t_{i}^{n}:= 
\frac{t_i}{\sum_{j=0}^{n}t_j}.\] 
\begin{obs}\label{obs}
 $ (1-t_{n+1}^{n+1})  
t_i^{n}=t^{n+1}_i $ and $ \lim_{n\rightarrow \infty}t_{i}^{n}=t_i $.
\end{obs}
\begin{proof}
\begin{align*}
\frac{t^{n+1}_i}{t_i^{n}}&=\frac{\sum_{j=0}^{n}t_j}{\sum_{j=0}^{n+1}t_j}=
\frac{\sum_{j=0}^{n+1}t_j-t_{n+1}}{\sum_{j=0}^{n+1}t_j}=1-t_{n+1}^{n+1},\\
\lim_{n\rightarrow \infty}t_{i}^{n}&=\lim_{n\rightarrow 
\infty}\frac{t_i}{\sum_{j=0}^{n}t_j}=\frac{t_i}{\lim_{n\rightarrow 
\infty}\sum_{j=0}^{n}t_j}=t_i.
\end{align*}
\end{proof}

We shall define recursively $ S_i^{n}\in \mathcal{S} $ for $ i,n\in \mathbb{N} $ with $ i\leq n $ in such a way that 
\begin{enumerate}[label=(\roman*)]
\item\label{item 1} $ C=\bigsqcup_{i\leq n} S_i^{n} $ for every $ n $;
\item\label{item 2} $ \mu_i(S_i^{n})\geq t_i^{n} $;
\item\label{item 3} For every fixed $ i\in \mathbb{N} $ the sequence 
$ (S_{i}^{n})_{n\geq i} $ is $ \subseteq $-decreasing.
\end{enumerate}
Observe that conditions \ref{item 1} and \ref{item 2} say  that for each fixed $ n $ the sets $S_0^{n}, S_1^{n},\dots, S_n^{n} 
$ form a fair division with respect to 
the 
measures $ \mu_i $ and entitlements $ t_i^{n} $. Although such a fair division can be found for every particular $ n $, it cannot 
be  guaranteed without condition \ref{item 3}  that they have a meaningful ``limit'' 
which provides a fair division  in the original settings.

We let $ S_0^{0}:=C $  which obviously satisfies the conditions. Suppose that 
$ S_{0}^{n}, S_{1}^{n}\dots, S_{n}^{n} $ are already defined for some $ n\in \mathbb{N} $. We need to find  for each $ i \leq 
n $ 
an $ S_{i}^{n+1}\subseteq S_i^{n} $ with $ \mu_i(S_i^{n+1})\geq t_i^{n+1} $ in such a way that for \[ 
S_{n+1}^{n+1}:=C\setminus 
\bigcup_{i\leq n}S_{i}^{n+1}  \]
we have  $ \mu_{n+1}(S_{n+1}^{n+1})\geq t_{n+1}^{n+1} $. For the last inequality it is enough to ensure that

\begin{equation}\label{ineq2}
\mu_{n+1}(S_i^{n}\setminus S_{i}^{n+1})\geq \mu_{n+1}(S_{i}^{n})\cdot t_{n+1}^{n+1} \text{ for  } i\leq n.
\end{equation}

 Indeed, since \[S_{n+1}^{n+1}=\bigsqcup_{i\leq n}  S_i^{n}\setminus S_{i}^{n+1}, \] 
the inequalities (\ref{ineq2}) imply \begin{align*}
\mu_{n+1}(S_{n+1}^{n+1})&=\mu_{n+1}\left( \bigsqcup_{i\leq n}  S_i^{n}\setminus 
 S_{i}^{n+1}\right)=\sum_{i=0}^{n} \mu_{n+1}(S_i^{n}\setminus S_{i}^{n+1})
 \geq \sum_{i=0}^{n}  
 \mu_{n+1}(S_{i}^{n})\cdot t_{n+1}^{n+1}\\ &=t_{n+1}^{n+1}\cdot \sum_{i=0}^{n}  
  \mu_{n+1}(S_{i}^{n})=  t_{n+1}^{n+1}\cdot  \mu_{n+1}(C)=t_{n+1}^{n+1}\cdot  1=t_{n+1}^{n+1},
\end{align*}  
 where we used \ref{item 1}  combined with the fact that $ \mu_{n+1} $ is a probability 
 measure. Therefore it is enough  to find for every $ i\leq n $ an $ S_i^{n+1}\subseteq S_i^{n}$  such that    
\begin{align}
 \mu_i(S_i^{n+1})&\geq t_i^{n+1} \label{ineq1}\\
 \mu_{n+1}(S_i^{n}\setminus S_{i}^{n+1})&\geq \mu_{n+1}(S_{i}^{n})\cdot 
  t_{n+1}^{n+1}. \label{ineq3}
\end{align}  
Let $ i\leq n $ be fixed. If $ \mu_{n+1}(S_i^{n})=0 $, then we let  $ S_i^{n+1}:=S_i^{n} $ which is clearly appropriate since $  
t_i^{n}\geq t_i^{n+1}$ (see Observation \ref{obs}). 
Suppose that $ \mu_{n+1}(S_i^{n})>0 $ and note that $ \mu_i(S_i^{n})  \geq t_i^{n}>0$ by assumption.     
We claim that choosing $ S_{i}^{n+1} $ to be  the slice corresponding to $ i $ in a fair division  of $ S_i^{n} $ between $ P_i $ 
and $ 
P_{n+1} $ with respect to the restrictions of $ 
\frac{\mu_i}{\mu_i(S_i^{n})} $ and  $ 
\frac{\mu_{n+1}}{\mu_{n+1}(S_i^{n})} $ to $\mathcal{S}\cap\mathcal{P}(S_i^{n}) $ and  respective entitlements $ 
1-t_{n+1}^{n+1} $ and $ 
t_{n+1}^{n+1} $ is suitable. 
Indeed, by the fairness of the obtained bipartition $ \{S_{i}^{n+1},\ S_{i}^{n}\setminus S_{i}^{n+1}  \} $ of $ 
S_{i}^{n} $  we have
\begin{align*}
\frac{\mu_i(S_{i}^{n+1})}{\mu_i(S_i^{n})} &\geq 1-t_{n+1}^{n+1},\\
 \frac{\mu_{n+1}(S_i^{n}\setminus S_{i}^{n+1})}{\mu_{n+1}(S_i^{n})}& \geq t_{n+1}^{n+1}.
\end{align*}
Here the second inequality  is equivalent with (\ref{ineq3}) and the first one implies  
(\ref{ineq1}) since
\[ \mu_i(S_{i}^{n+1})\geq (1-t_{n+1}^{n+1})  \mu_i(S_i^{n})\geq (1-t_{n+1}^{n+1})  t_i^{n}=t^{n+1}_i, \]
where we used $ \mu_i(S_i^{n})\geq t_i^{n}  $ and Observation \ref{obs}. The recursion is done. 

We define $ S_i:=\bigcap_{n \geq i}S_i^{n}  $ for $ i\in \mathbb{N} $. Then 
for 
$ i<j $ we have $ S_i\cap S_j=\emptyset $ because $ S_i \subseteq S_{i}^{j},\  S_j \subseteq S_{j}^{j} $ and $ S_{i}^{j}\cap 
S_{j}^{j}=\emptyset $ by \ref{item 1}. Furthermore,
\[ \mu_i(S_i)=\mu_i\left( \bigcap_{n \geq i}S_i^{n}\right)=\lim_{n\rightarrow \infty}\mu_i(S_i^{n})\geq\lim_{n\rightarrow 
\infty}t_{i}^{n}=t_i   \]
by \ref{item 3},  \ref{item 2}  and Observation \ref{obs}. This completes the proof of Theorem \ref{t: main cake}. 
\end{proof}
 

\end{document}